\documentclass[11pt]{article}
\usepackage[top=1in, bottom=1in, left=1.25in, right=1.25in]{geometry}
\usepackage{amsmath,amssymb,amsthm,xcolor}
\newtheorem{theorem}{Theorem}[section]
\newtheorem{corollary}[theorem]{Corollary}

\newtheorem{lemma}[theorem]{Lemma}
\newtheorem{example}[theorem]{Example}
\newtheorem{remark}[theorem]{Remark}
\newcommand{\MB}[1]{\ensuremath{\mathbf{#1}}}
\newcommand{\MC}[1]{\ensuremath{\mathcal{#1}}}

\newcommand{\B}{\ensuremath{\mathcal{B}}}
\title{An Explication of Optimal Equidistant Codes\footnote{This paper is based on the talk entitled ``Optimal Equidistant Codes--A Detective Story'' that I presented at the 2026 Summer Canadian Mathematical Society meeting on June 7, 2026.}}
\author{Douglas R.\ Stinson\\David R.\ Cheriton School of Computer Science\\
University of Waterloo\\ Waterloo ON, N2L 3G1, Canada}

\begin{document}
\maketitle

\begin{abstract}
We discuss the problem of characterizing equidistant binary codes of a given length $n$ having largest possible distance and the maximum number of codewords. Such characterizations have been studied by several authors over the years and they involve symmetric BIBDs with certain parameters. In this primarily expository paper, we investigate the history of this problem and give a unified presentation of the main results. Perhaps surprisingly, researchers on this problem were unaware of early relevant work by Marrero and Butson \cite{MB} from 1973. Also, it turns out that published results on characterizations of equidistant binary codes have missed one of the possible subcases when $n \equiv 2 \bmod 4$.
\end{abstract}

\section{Introduction}
\label{intro.sec}

An \emph{equidistant code} $E(n,d,m)$ consists of $m$ binary codewords of length $n$, such that the hamming distance between any two distinct codewords is exactly $d$. An \emph{equidistant constant-weight code} $EC(n,d,w,m)$ consists of $m$ binary codewords of length $n$ and weight $w$, such that the distance between any two distinct codewords is exactly $d$.

Equidistant codes 
have been studied at least since the late 1960's. The 1968 paper by Semakov and Zinoviev \cite{Semakov-Zinoviev} is perhaps the first one to study equidistant codes over a $q$-ary alphabet.\footnote{Some work done on $q$-ary equidistant codes will be mentioned in Section \ref{related}.}
 Among other things, this paper proves some equivalences involving BIBDs. 

One basic problem involving equidistant binary codes is to maximize $m$ as a function of $d$ and $n$. In this note, we discuss a more specialized problem, which can be described in three steps; see Figure \ref{fig1}. The codes 
$E(n,d^*,m^*)$ will be termed \emph{optimal} equidistant binary codes.
Much work on this problem has been done, an almost complete solution is found in Ionin and Shrikhande \cite{IS95}.
In this note, we will provide a somewhat simplified solution, based on \cite{IS95,MB,SvR,vanLint84} and other related papers. Some other relevant references on equidistant codes include \cite{FKLW,Heg,IS95,JHall,SvR,vanLint73,vanLint84}.

\begin{figure}[tb]
\begin{center}
\begin{tabular}{|c|p{4.5in}|}
\hline
step 1 & Given $n$ (the length of the codewords), determine the maximum $m$ such an $E(n,d,m)$ exists for some $d$.
Denote this maximum value of $m$ (which of course is a function of $n$) by $m^*$.\\ \hline
step 2 & Determine the maximum $d$ such that an $E(n,d,m^*)$ exists. Denote this maximum value of $d$ (which is also  a function of $n$) by $d^*$.\\ \hline
step 3 &  Give a strong characterization of  $E(n,d^*,m^*)$. \\ \hline
\end{tabular}\end{center}
\caption{Optimal Equidistant Binary Codes}
\label{fig1}
\end{figure}


The first two steps of this three-step process are fairly simple. Step 1 in Figure \ref{fig1} is basically an application of a well-known classical Fisher-type inequality for $(r,\lambda)$-designs, which has been known since 1966. We obtain the following (see Section \ref{step1.sec} for the proof).

\begin{theorem}
\label{step1}
If $n \equiv 3 \bmod 4$, then $m^* \leq n+1$.
If $n \not\equiv 3 \bmod 4$, then $m^* \leq n$. 
\end{theorem}

Step 2 in  Figure \ref{fig1} can be solved as an easy application of the 1960 Plotkin Bound for binary codes (see Section \ref{step2.sec}). 

\begin{theorem}
\label{step2}
Suppose $n \geq 3$.  If $n \equiv 3 \bmod 4$, then $d^* \leq (n+1)/2$. If $n \equiv r \bmod 4$, where $0\leq r \leq 2$, then 
$d^* \leq (n-r)/2$.
\end{theorem}

The third step in  Figure \ref{fig1} is more complicated; we will discuss it in Section \ref{step3.sec}.
The following theorem summarizes the combinatorial characterizations of optimal equidistant codes that we will prove.

\begin{theorem}
\label{step3}
\mbox{\quad} 
\begin{enumerate}
\item For $n \equiv 3 \bmod 4$, an $E(n,(n+1)/2,n+1)$ is equivalent to a Hadamard matrix of order $n+1$.
\item For $n \equiv 0 \bmod 4$, an  $E(n,n/2,n)$ is equivalent to a Hadamard matrix of order $n$.
\item For $n \equiv 1 \bmod 4$, an $E(n,(n-1)/2,n)$ is equivalent to a 
a $(2u^2+2u+1, u^2, (u^2 -u)/2)$-SBIBD, where $n = 2u^2+2u+1$.
\item For $n \equiv 2 \bmod 4$, an $E(n,(n-2)/2,n)$ is equivalent to a 
$(12u^2+8u+2, 6u^2 + u, 3u^2 -u)$-SBIBD  where $n = 12u^2 + 8u + 2$, or
a $(12u^2 - 8u + 2,6u^2 -7u+2,3u^2 - 5u+2)$-SBIBD where $n= 12u^2 - 8u + 2$
\end{enumerate}
\end{theorem}

\begin{remark}
{\rm 
Different  parts of Theorem \ref{step3} were proven independently in various papers, often using different mathematical language.
For example, part 1 was shown in \cite{Cordes}, part 2 is proven in \cite{MRS} and
part 3 was proven in \cite{SvR,vanLint84}.
The most complete statement of these results was given by Ionin and Shrikhande \cite[Theorem 4.2]{IS95}. However, the second subcase of part 4 is not included in \cite[Theorem 4.2]{IS95}; this is apparently an oversight. We will discuss the history of these results in more detail later in the paper.}
\end{remark}

\section{Preliminaries}

In the rest of this section, we present some well-known equivalent formulations of equidistant codes.
The following result  has been used by many authors.

\begin{lemma}[folklore]
\label{weight.thm}
There exists an $E(n,d,m+1)$ if and only if there exists an
$EC(n,d,d,m)$.
\end{lemma}
\begin{proof}
Let $\MC{C}$ be an $E(n,d,m+1)$ and 
let $\MB{c_0} \in \MC{C}$. Define
\[ \MC{C}' = \{ \MB{c} -\MB{c_0}: \MB{c} \in \MC{C}, \MB{c} \neq \MB{c_0} \} .\]
Then $\MC{C}'$ is an $EC(n,d,d,m)$.

Conversely, suppose that $\MC{C}'$ is an $\MC{C}'$ is an $EC(n,d,d,m)$.
Define 
\[ \MC{C} = \MC{C}' \cup \{ \MB{0}\},\]
where $\MB{0}$ is a vector of $n$ zeroes. Clearly $\MC{C}$ is an $E(n,d,m+1)$.
\end{proof}

Let 
\[ A(n,d) = \max \{ m : \text{an $E(n,d,m)$ exists} \} \]
and let
\[ AC(n,d,w) = \max \{ m : \text{an $EC(n,d,w,m)$ exists} \} .\]

The following is an immediate consequence of Lemma \ref{weight.thm}.

\begin{corollary}
$A(n,d) = AC(n,d,d)+1$.
\end{corollary}

\begin{lemma}[folklore]
\label{setsystem.lem}
An $EC(n,d,d,m)$ with $m \geq 2$ is equivalent to a set system having $n$ points, $m$ blocks, block size $d$ and constant block intersection $d/2$.
\end{lemma}
\begin{proof}
Given an  $EC(n,d,d,m)$, treat each codeword as the characteristic vector of a subset of an $n$-set.
\end{proof}

\begin{remark}
{\rm An immediate consequence of Lemma \ref{setsystem.lem} is that $d$ is even.} 
\end{remark}

We will make use of some results concerning $(r,\lambda)$-designs. 
An \emph{$(r,\lambda)$-design} is a set system in which every point occurs in exactly $r$ blocks and every pair of points occurs in exactly $\lambda$ blocks. To avoid trivial cases, it is also required that $r > \lambda$. Note that blocks of size one are permitted in an $(r,\lambda)$-design.
The first paper of which I am aware that contains an explicit definition of $(r,\lambda)$-designs is the 1966 paper of Stanton and Mullin \cite{StMu66}.

\begin{lemma}[folklore]
\label{dual.thm}
The dual of a set system having $n$ points, $m$ blocks, block size $d$ and constant block intersection $d/2$ is an $(r,\lambda)$-design on $m$ points and $n$ blocks with $r = d$ and $\lambda = d/2$.
\end{lemma} 

\begin{remark}
{\rm The dual of a set system is obtained by transposing its incidence matrix, so that points and blocks are interchanged.}
\end{remark}

From Lemmas \ref{weight.thm}, \ref{setsystem.lem} and \ref{dual.thm},  we  have the following equivalences.

\begin{theorem}
\label{equivalence}
Suppose $m \geq 2$.
Existence of the following are equivalent.
\begin{enumerate}
\item An $E(n,d,m+1)$.
\item An $EC(n,d,d,m)$.
\item A set system having $n$ points, $m$ blocks of size $d$ and constant block intersection $d/2$.
\item An $(r,\lambda)$-design on $m$ points and $n$ blocks with $r = d$ and $\lambda = d/2$.
\end{enumerate}
\end{theorem}

\medskip
We will later make use of a technique known as \emph{$r$-complementation} (see  Vanstone \cite{Vanstone}). Suppose we have an $(r,\lambda)$-design $(X,\B)$, where $X$ is a set of $m$ points and $\B$ is a collection of $n$ blocks.
Pick any point $x\in X$. For every block $B \in \B$ with $x \in B$, replace $B$ by $X \setminus B$. 
For every block $B \in \B$ with $x \not\in B$, leave $B$ unchanged. The resulting set of blocks is denoted by $\B'$. Also, let $X' = X \setminus \{x\}$. (Note that no block in $\B'$ contains the point $x$.) The set system $(X',\B')$ is an \emph{$r$-complement} of $(X,\B)$.

The following result is proven by a simple counting argument.

\begin{lemma}\textup{\cite{Vanstone}}
\label{r-comp}
Any $r$-complement of an $(r,\lambda)$-design on $m$ points and $n$ blocks is a $(2s,s)$-design having  
$m-1$ points and $n$ blocks, where $s= r - \lambda$. 
\end{lemma}

The following is an immediate consequence of Theorem \ref{equivalence} and Lemma \ref{r-comp}.

\begin{corollary}\label{rlam-equiv}
The $r$-complement of an $(r,\lambda)$-design on $m$ points and $n$ blocks is 
equivalent to an  $E(n,2(r - \lambda),m)$.
\end{corollary}

\begin{proof}
The $r$-complement of an $(r,\lambda)$-design on $m$ points and $n$ blocks is $(2s,s)$-design having  
$m-1$ points and $n$ blocks, where $s= r - \lambda$. From $4. \Leftrightarrow 1.$ in Theorem \ref{equivalence}, the $(2s,s)$-design is equivalent to a $E(n,2(r - \lambda),m)$.
\end{proof}

\begin{example}
\label{example13}
{\rm
We start with a $(13,4,1)$-SBIBD having the following blocks:
\[
\begin{array}{llll}
\{ 0,1,3,9\}  & \{ 1,2,4,10\} & \{ 2,3,5,11\} & \{ 3,4,6,12\} \\
\{ 4,5,7,0\}  & \{ 5,6,8,1\} & \{ 6,7,9,2\} & \{ 7,8,10,3\} \\
\{ 8,9,11,4\}  & \{ 9,10,12,5\} & \{ 10,11,0,6\} & \{ 11,12,1,7\} \\
\{ 12,0,2,8\}
\end{array}
\]
Suppose we complement the four blocks containing $0$:
\[
\begin{array}{llll}
\{ 2,4,5,6,7,8,10,11,12\}  & \{ 1,2,4,10\} & \{ 2,3,5,11\} & \{ 3,4,6,12\} \\
\{ 1,2,3,6,8,9,10,11,12\}  & \{ 5,6,8,1\} & \{ 6,7,9,2\} & \{ 7,8,10,3\} \\
\{ 8,9,11,4\}  & \{ 9,10,12,5\} & \{ 1,2,3,4,5,7,8,9,12\} & \{ 11,12,1,7\} \\
\{ 1,3,4,5,6,7,9,10,11\}
\end{array}
\]
We obtain a $(6,3)$-design having 12 points and 13 blocks (Lemma \ref{r-comp}).

By Corollary \ref{rlam-equiv}, this design is equivalent to an $E(13,6,13)$.
We explicitly construct the equidistant code.
The $12$ by $13$ point-block incidence matrix of the $(6,3)$-design is
\[
\left(
\begin{array}{ccccccccccccc}
0 & 1 & 0 & 0 & 1 & 1 & 0 & 0 & 0 & 0 & 1 & 1 & 1 \\ 
1 & 1 & 1 & 0 & 1 & 0 & 1 & 0 & 0 & 0 & 1 & 0 & 0 \\ 
0 & 0 & 1 & 1 & 1 & 0 & 0 & 1 & 0 & 0 & 1 & 0 & 1 \\ 
1 & 1 & 0 & 1 & 0 & 0 & 0 & 0 & 1 & 0 & 1 & 0 & 1 \\ 
1 & 0 & 1 & 0 & 0 & 1 & 0 & 0 & 0 & 1 & 1 & 0 & 1 \\ 
1 & 0 & 0 & 1 & 1 & 1 & 1 & 0 & 0 & 0 & 0 & 0 & 1 \\ 
1 & 0 & 0 & 0 & 0 & 0 & 1 & 1 & 0 & 0 & 1 & 1 & 1 \\ 
1 & 0 & 0 & 0 & 1 & 1 & 0 & 1 & 1 & 0 & 1 & 0 & 0 \\ 
0 & 0 & 0 & 0 & 1 & 0 & 1 & 0 & 1 & 1 & 1 & 0 & 1 \\ 
1 & 1 & 0 & 0 & 1 & 0 & 0 & 1 & 0 & 1 & 0 & 0 & 1 \\ 
1 & 0 & 1 & 0 & 1 & 0 & 0 & 0 & 1 & 0 & 0 & 1 & 1\\ 
1 & 0 & 0 & 1 & 1 & 0 & 0 & 0 & 0 & 1 & 1 & 1 & 0   
\end{array}
\right).
\]
The columns of this matrix correspond to the blocks of the $(6,3)$-design and 
the rows of this matrix form an $EC(13,6,6,12)$. 
If we adjoin a codeword consisting of $13$ $0$'s, then we obtain
an $E(13,6,13)$. In view of Theorem \ref{step3}, this equidistant code is optimal.
}\hfill$\blacksquare$
\end{example}


\section{Proof of Theorem \ref{step1}}
\label{step1.sec}

Theorem \ref{step1} is an immediate consequence of classical results.
Fisher's Inequality for BIBDs states that a $(v,b,r,k,\lambda)$-BIBD must have $b \geq v$. 
The 1966 Stanton-Mullin paper \cite{StMu66} extends Fisher's Inequality (for BIBDs) to  
$(r,\lambda)$-designs.

\begin{theorem}
\label{Fisher}
\textup{\cite{StMu66}}
If an $(r,\lambda)$-design has $m$ points and $n$ blocks, then $n \geq m$. 
\end{theorem}

The proof of Theorem \ref{Fisher} given in \cite{StMu66} builds on the following result of Ryser \cite{Ryser1950} (proven in 1950) that applies specifically to the case $n=m$. 

\begin{theorem}
\label{Ryser}
\textup{\cite{Ryser1950},\cite[Ch. 8, \S2]{Ryser}}
If an $(r,\lambda)$-design has $n$ points and $n$ blocks, then it is a (symmetric)  $(n,n,r,r,\lambda)$-BIBD. 
\end{theorem}

\begin{remark}
{\rm Fisher's Inequality for $(r,\lambda)$-designs can also be proven directly. Most proofs of Fisher's Inequality (for BIBDs) also apply without modification to $(r,\lambda)$-designs, because the proofs do not make use of the fact that, in a BIBD, the blocks all have the same size.}
\end{remark}

In step 1 of Figure \ref{fig1}, we are interested in maximizing $m$ as a function of $n$.
By Fisher's inequality for $(r,\lambda)$-designs (Theorem \ref{Fisher}), 
we have $n \geq m$.  Hence, from Theorem \ref{weight.thm}, we immediately obtain the following bound, which holds for all possible values of $d$.
\begin{corollary}
\label{cor2}
$A(n, d) \leq n+1$.
\end{corollary}

\begin{remark}
{\rm It has been observed by Fu {\it et al.} \cite{FKLW} and by Ionin and Shrikhande \cite{IS95} that Corollary \ref{cor2} is in fact a special case of a much more general result proven by Delsarte  \cite[Theorem 4.8]{Delsarte} in 1973 using more sophisticated mathematical techniques.}
\end{remark}




Let us  consider the case $n = m$ in more detail. From Theorem \ref{equivalence},  an $E(n,d,n+1)$ is equivalent to 
an $(r,\lambda)$-design on $n$ points and $n$ blocks with $r = d$ and $\lambda = d/2$.
From Theorem \ref{Ryser}, this $(r,\lambda)$-design must be an $(n,r,\lambda)$-SBIBD. Therefore 
$(d/2)(n-1) = d(d-1)$ and hence
$n = 2d-1$. Since $d$ is even, it is necessary that $n \equiv 3 \bmod 4$. Consequently, if $n \not\equiv 3 \bmod 4$, it must be the case that $n \geq m+1$. 

The above discussion proves Theorem \ref{step1}. But we can say a bit more. The $(n,r,\lambda)$-SBIBD is  a $(2d-1,d,d/2)$-SBIBD where $d$ is even.
Writing $d = 2t$, we have a $(4t-1,2t,t)$-SBIBD. This SBIBD is equivalent to a Hadamard matrix of order $4t$. Thus we obtain the following theorem. 
\begin{theorem}
\label{3mod4}
If $n \equiv 3 \bmod 4$, then 
$m^* = n+1$, $d^* = (n+1)/2$ and an $E(n,(n+1)/2,n+1)$ is equivalent to a Hadamard matrix of order $n+1$.
\end{theorem}

\begin{remark}
{\rm Theorem \ref{3mod4} was stated explicitly by Ionin and Shrikhande in  \cite[Theorem 2.1]{IS95}, but it is equivalent to an earlier result of Mullin, Roy and Schellenberg that was presented in \cite[Theorem 4.6]{MRS}. Note that the results in \cite{MRS} are described in the setting of the so-called ``Cordes problem'' that was introduced in \cite{Cordes}. See Section \ref{related} for additional discussion on the Cordes problem.}
\end{remark}



Theorem \ref{3mod4} completely solves all three steps for the case $n \equiv 3 \bmod 4$.
Therefore, in the next section, we only have to consider $n \equiv 0,1,2 \bmod 4$.

\section{Proof of Theorem \ref{step2}}
\label{step2.sec}




If $n \not\equiv 3 \bmod 4$, then we have shown that $m^* = n$ in Theorem \ref{step1}. 
In these cases, we can use the classical Plotkin Bound for binary codes to upper-bound $d$.\footnote{A similar approach has  been used  to study equidistant $q$-ary codes. See, e.g., Fu {\it et al.} \cite[Corollary 2]{FKLW} and 
Sinha, Wang and Wu \cite{Sinha}.} We just need the case where $d$ is even.

\begin{theorem}[Plotkin Bound]\textup{\cite[Theorem 1]{Plotkin}}
\label{plotkin.thm}
If a binary code of length $n$ has $m$ codewords and minimum distance $d > n/2$, then
$m \leq  \frac{2d}{2d-n} $.
\end{theorem}

We can now prove Theorem \ref{step2}. It suffices to prove that,
if $n \equiv r \bmod 4$, where $0\leq r \leq 2$, then 
$d^* \leq (n-r)/2$.

\begin{proof}[Proof of Theorem \ref{step2}]
The case $n \equiv 3 \bmod 4$ was already done in Theorem \ref{3mod4}, so we can assume $m = n$,  $n\equiv 0,1,2 \bmod 4$ and $n \geq 4$.
Write $n = 4t + r$, where $0 \leq r \leq 2$. We claim that $d \leq 2t$. Suppose not; then $d \geq 2t+2$ because $d$ is even. Note that $d > n/2$ in this case. Then Theorem \ref{plotkin.thm} asserts that
\begin{align*} n &\leq  \frac{2d}{2d-n} \\
&= 1 + \frac{n}{2d-n} \\
&\leq 1 + \frac{n}{4t+4-(4t+r)} \\
&=  1 + \frac{n}{4-r} \\\
&\leq 1 + \frac{n}{2}.\end{align*}
This is clearly a contradiction because $n \geq 4$.
\end{proof}

\section{Proof of Theorem \ref{step3}}
\label{step3.sec}

Now we want to characterize the optimal equidistant codes.
We have already done this for $n \equiv 3 \bmod 4$ in Theorem \ref{3mod4}.
Therefore we only need to address $n \equiv 0,1,2 \bmod 4$. We will be discussing the results stated as cases 2, 3 and 4 of Theorem \ref{step3}.

The case $n \equiv 0 \bmod 4$ is the easiest. 

\begin{lemma}
\label{0mod4}
If $n \equiv 0 \bmod 4$, then an $E(n,n/2,n)$ is equivalent to a Hadamard matrix of order $n$.
\end{lemma}

\begin{proof}
Suppose we have an $E(n,n/2,n)$. Replace all $0$'s by $-1$'s. Then the resulting $n$ by $n$ matrix is easily seen to be a Hadamard matrix of order $n$. 
Conversely,  a Hadamard matrix of order $n$ yields an $E(n,n/2,n)$ if we replace all $-1$'s
 by $0$'s.
 \end{proof}
 \begin{remark}
 {\rm Similar arguments can be found in Cordes \cite[Theorem 3]{Cordes} and Marrero and Butson \cite[Theorem 3.2]{MB}.
 }
 \end{remark}

The cases  $n \equiv 1,2 \bmod 4$ are more difficult, but they can  be handled using results of Marrero and Butson \cite{MB,Marrero}, 
who studied ``pseudo $(v,k,\lambda)$-designs'' in the 1970's. A \emph{pseudo $(v,k,\lambda)$-design} is a set system having $n$ points, $n-1$ blocks of size $k$ and constant block intersection $\lambda$. Of course the dual
of a pseudo design is an $(r, \lambda)$-design with $r = k$ and $n = m+1$ (i.e., the number of blocks is one more than the number of points). When $k = 2 \lambda$, they are examples of the designs considered in  Theorem \ref{equivalence}. 

The following important result 
concerning the case $k = 2 \lambda$ is proven by Marrero and Butson in \cite{MB} (see also the survey by Marrero \cite[Theorem 3.2]{Marrero}).

\begin{theorem}
\textup{\cite{MB}}
\label{MB.thm}
\mbox{\quad}
\begin{enumerate}
\item A pseudo $(4\lambda ,2\lambda,\lambda)$-design is equivalent to a
$(4\lambda -1,2\lambda-1,\lambda-1)$-SBIBD.
\item A pseudo $(v ,2\lambda,\lambda)$-design with $v \neq 4 \lambda$ is equivalent to a
$\left(v,\frac{v-D}{2},\frac{v-D}{2} - \lambda \right)$-SBIBD, where 
\begin{equation}
\label{D.eq}
D^2 = v^2 - 4\lambda(v-1).
\end{equation}
\end{enumerate}
\end{theorem}

\begin{remark}
{\rm We can assume without loss of generality that $D \geq 0$ in Theorem \ref{MB.thm}, replacing $D$ by $-D$ if $D < 0$.}
\end{remark}
 
 Part 1 of Theorem \ref{MB.thm} provides another solution to  the case $n \equiv 0 \bmod 4$, which we have already discussed in Lemma \ref{0mod4}.
 Part 2 of Theorem \ref{MB.thm} is relevant to the cases $n \equiv 1,2 \bmod 4$. Marrero and Butson did not explicitly compute the parameters of the relevant SBIBDs in part 2. We summarize the results of calculations that are found in \cite{SvR} for $n \equiv 1 \bmod 4$ and in \cite{IS95} for $n \equiv 2 \bmod 4$.
 
 First, suppose $n \equiv 1 \bmod 4$. 
An $E(n,(n-1)/2,n)$ is the same thing as a  pseudo $(n ,(n-1)/2,(n-1)/4)$-design. By Theorem \ref{MB.thm}, this 
pseudo design is equivalent to a 
$\left(n,\frac{n-D}{2},\frac{n-D}{2} - \frac{n-1}{4} \right)$-SBIBD, where 
\begin{align*}
D^2 &= n^2 - 4\times \frac{n-1}{4}  \times (n-1) \\&= n^2 - (n-1)^2 \\&= 2n-1.\end{align*}
Write $n = 4t+1$; then $D^2 = 8t+1$.
Clearly $D$ must be odd, so write $D = 2u+1$.
Then $4u^2 + 4u + 1 = 8t+1$, so $t = (u^2 + 2)/2$
and $n = 2u^2 + 2u + 1$. We also have
$(n-D)/2 = u^2$ and $(n-D)/2 - (n-1)/4 = u^2 - (u^2 + u)/2 = (u^2-u)/2$.
This proves part 3 of Theorem \ref{step3}.

Now we turn to the case  $n \equiv 2 \bmod 4$. 
An $E(n,(n-2)/2,n)$ is the same thing as a  pseudo $(n ,(n-2)/2,(n-2)/4)$-design.
By Theorem \ref{MB.thm}, this 
pseudo design is equivalent to a 
$\left(n,\frac{n-D}{2},\frac{n-D}{2} - \frac{n-2}{4} \right)$-SBIBD, where 
\begin{align*}
D^2 &= n^2 - 4\times \frac{n-2}{4}  \times (n-1) \\&= n^2 - (n-1)(n-2) \\&= 3n-2.\end{align*}
Write $n = 4t+2$; then $D^2 = 12t+4$.
Therefore $D \equiv 2,4 \bmod 6$, so we have two subcases to consider. 
(In \cite{IS95}, only the case $D \equiv 2 \bmod 6$ is addressed.)

First, suppose that $D = 6u+2$. Then
$t = (D^2-4)/12 = 3u^2 + 2u$. By part 2 of Theorem \ref{MB.thm}, the
pseudo $(4t+2 ,2t,t)$-design is equivalent to a
$\left(12u^2 + 8u + 2,6u^2 + u,3u^2 - u \right)$-SBIBD.

In the second subcase, where $D = 6u-2$, 
we obtain $t = 3u^2 - 2u$. By part 2 of Theorem \ref{MB.thm}, the
pseudo $(4t+2 ,2t,t)$-design is equivalent to a
\begin{center}
$\left(12u^2 - 8u + 2,6u^2 -7u+2,3u^2 - 5u+2 \right)$-SBIBD.
\end{center}

In both cases, the parameters of the SBIBD must satisfy the conditions of the
Bruck-Ryser-Chowla Theorem. For a $(v,k,\lambda)$-SBIBD to exist with $v$ even, $k - \lambda$ must be a perfect square. When $D = 6u+2$, a list of possible parameters of the SBIBD are described in \cite{IS95}, using a recurrence relation. The smallest nontrivial parameter set for which the SBIBD exists is 
$(66,26,10)$ (corresponding to $u = 2$).
When $D = 6u-2$, the smallest parameter set where $k-\lambda$ is a perfect square is when $u = 9$; then $3u^2 - 2u = 225 = 15^2$. The corresponding symmetric BIBD  would be a $(902, 425,200)$-SBIBD; however,  I do not know if this BIBD exists.

\section{Discussion and Related Work}
\label{related}

In cases 3 and 4 of Theorem \ref{step3}, the desired equidistant codes can all be constructed from the hypothesized SBIBDs by applying Corollary \ref{rlam-equiv}, as illustrated in Example \ref{example13}. 
The {difficult} direction of the equivalence proof for {$n \equiv 1,2 \bmod 4$}  is to show that  optimal equidistant codes imply the existence of the stated SBIBDs.
 There are four different proofs in the literature:
\begin{enumerate}
\item The {Marraro-Butson} proof \cite{MB} involving pseudo designs.
\item The {Stinson-van Rees} proof \cite{SvR} for {$n \equiv 1 \bmod 4$}.
\item The {van Lint} proof \cite{vanLint84} for {$n \equiv 1 \bmod 4$}.
\item The {Ionin-Shrikhande} proof \cite{IS95}.
\end{enumerate}
All of these proofs are rather complicated. The Stinson-van Rees proof and the van Lint proof
are probably the shortest, but it is not clear if they can be generalized to cover the cases
{$n \equiv 2 \bmod 4$}.\footnote{I tried to generalize the proof from \cite{SvR} to handle this case, but I encountered  a problem in doing so.
}

The relevance of the work of Marrero and Butson \cite{MB,Marrero} to optimal equidistant codes has apparently not been noticed previously. Perhaps this is not too surprising, as no mention of equidistant codes can be found in \cite{MB,Marrero} and most (if not all) papers on equidistant codes do not reference \cite{MB,Marrero}. The one exception of which I am aware is the 1979 paper of McCarthy and Vanstone \cite{McV} on ``regular pairwise balanced designs.'' The McCarthy-Vanstone  paper does not study optimal equidistant codes, but it does mention the equivalence of equidistant codes and $(r,\lambda)$ designs. It also cites \cite{MB}, without going into detail about the contributions of the work by Marrero and Butson.

The 1984 papers by Stinson and van Rees \cite{SvR} and by van Lint \cite{vanLint84} both 
address equidistant codes of length $n \equiv 1 \bmod 4$.
John van Rees and I were motivated to study this case due to almost tight upper and lower bounds on 
$A(n,d^*)$ for $n \equiv 1 \bmod 4$ that were proven in \cite[Theorem 4.8]{MRS} in the context of the Cordes problem, which is described below 
(we were not aware at that time of the work of Marrero and Butson). The paper by van Lint \cite{vanLint84} was an alternate proof that was written as a result of a request by the editors of \emph{Combinatorica}, after we submitted \cite{SvR} for publication in that journal.

In a 1978 paper, Cordes \cite{Cordes} considered resolvable set systems in which each parallel class consists of $m$ blocks of size $n$. When $m = 2$, each parallel class consists of two blocks of size $n$. Cordes required that the number $\sigma(2,n)$ of pairs of points common to two parallel classes is  minimized. He showed that 
\begin{equation}
\label{sigma.eq} \sigma(2,n)  = s (2n - 2s - 2),
\end{equation}
where $s = \lfloor \frac{n}{2} \rfloor$.
When $n$ is even, (\ref{sigma.eq}) asserts that  $\sigma(2,n) = n(n-2)/2$. Equality occurs if and only if each block in one parallel class intersects each block in a different parallel class in $n/2$ points. Suppose we choose one block from each parallel class. Then we have a set of blocks of size $n$, chosen from a set of $2n$ points, such that any two blocks intersect in $n/2$ points. That is, we have a constant weight equidistant code. 

Cordes denoted the maximum number of rounds by $k(2,n)$. 
This is equivalent to maximizing the size of the associated equidistant code.
Indeed, he proved in 
\cite[Theorem 3]{Cordes}
that if $n$ is even, then $k(2,n) = 2n-1$ if and only if a Hadamard matrix of order $2n$ exists.
This is equivalent to part 2 of Theorem \ref{step3}.

Mullin, Roy and Schellenberg \cite{MRS} considered a more general version of the Cordes problem. In \cite[\S 4]{MRS}, rounds consisting of two not necessarily equal sized sets are considered. 
When the sizes of the two sets are $2\ell$ and $2\ell+1$ for an integer $\ell$, their results are relevant to equidistant codes.

Various papers have discussed connections between equidistant codes and projective planes; see, for example, \cite{Deza,JHall,vanLint73}. There are also  papers \cite{Deza,McV,vanLint73} that have proven inequalities involving the block sizes of $(r,\lambda)$-designs.

Work has also been done on $q$-ary equidistant codes; 
see \cite{BZT,FKLW,Semakov-Zinoviev,Sinha}. 
It was shown a long time ago by Semakov and Zinoviev \cite{Semakov-Zinoviev} that optimal $q$-ary equidistant codes (i.e., those whose parameters meet the $q$-ary Plotkin Bound) are equivalent to certain resolvable BIBDs. Related results can be found in \cite{Bassalygo-Zinoviev,Sinha}. 
Optimal $q$-ary equidistant codes with $d = 3$ and $4$ are classified in \cite{BZT}
and some computer searches and enumerations are reported in \cite{BT}.

A recent (2026) paper by Heged\"{u}s \cite{Heg} introduces equidistant codes as a ``new'' problem. The main theorem in \cite{Heg} (Theorem 13) is the Fisher-type (AKA Delsarte-type) inequality that we stated as  Corollary \ref{cor2}. Another recent derivation of the same inequality is found in \cite{Kom}. 
The paper \cite{Heg} also includes a conjecture (Conjecture 1) at the end of their paper. However, the truth of this conjecture follows immediately from Fu {\it et al.} \cite[Corollary 2]{FKLW}, for example; it is just an application of the $q$-ary Plotkin Bound. See \cite{HHY} for a discussion of similar results.

\end{document}